\documentclass[10pt]{amsart}
\usepackage[cp1251]{inputenc}
\usepackage{amstext}
\usepackage{amsfonts}
\usepackage{amssymb}
\usepackage{amsbsy}
\usepackage{amsmath}
\usepackage{latexsym}
\usepackage{xy}
\usepackage{hhline}
\xyoption{all}
\newcommand{\vtx}[1]{*+[o][F-]{\scriptscriptstyle #1}} 

\newcounter{num}[section] %

\newenvironment{theo}
{\refstepcounter{num}%
\bigskip\noindent{\bf Theorem~\arabic{section}.\arabic{num}. }\it}

\newenvironment{cor}
{\refstepcounter{num}%
\bigskip\noindent{\bf Corollary~\arabic{section}.\arabic{num}. }\it}

\newenvironment{lemma}
{\refstepcounter{num}%
\bigskip\noindent{\bf Lemma~\arabic{section}.\arabic{num}. }\it}

\newenvironment{example}
{\refstepcounter{num}%
\bigskip\noindent{\bf Example~\arabic{section}.\arabic{num}.}}

\newenvironment{conj}
{\refstepcounter{num}%
\bigskip\noindent{\bf Conjecture~\arabic{section}.\arabic{num}. }\it}

\newenvironment{remark}
{\refstepcounter{num}%
\bigskip\noindent{\bf Remark~\arabic{section}.\arabic{num}.}}

\newcommand{\definition}[1]
{\refstepcounter{num}%
\bigskip\noindent{\bf Definition~\arabic{section}.\arabic{num}}~({\it #1}).}

\newcommand{\defin}{\refstepcounter{num}%
\bigskip\noindent{\bf Definition~\arabic{section}.\arabic{num}.}}

\newcommand{\Ref}[1]{(\ref{#1})}

\newcounter{thepic}

\newenvironment{eq}{\begin{equation}}{\end{equation}}

\newcommand{\si}{\sigma}
\newcommand{\al}{\alpha}
\newcommand{\be}{\beta}

\newcommand{\la}{\lambda}
\newcommand{\de}{\delta}

\newcommand{\LA}{\langle}
\newcommand{\RA}{\rangle}

\newcommand{\tr}{\mathop{\rm tr}}

\newcommand{\degm}[1]{\mathop{\rm deg^{-}}(#1)} 
\newcommand{\degp}[1]{\mathop{\rm deg^{+}}(#1)} 

\newcommand{\Char}{\mathop{\rm char}}
\newcommand{\sign}{\mathop{\rm{sgn }}}

\newcommand{\Hom}{{\mathop{\rm{Hom }}}}

\newcommand{\Sp}{S\!p}

\newcommand{\FF}{{\mathbb{F}}}   
\newcommand{\NN}{{\mathbb{N}}}
\newcommand{\ZZ}{{\mathbb{Z}}}   

\newcommand{\Q}{\mathcal{Q}}    
\newcommand{\n}{\boldsymbol{n}} 
\newcommand{\m}{\boldsymbol{m}} 
\renewcommand{\i}{\boldsymbol{i}} 
\newcommand{\D}{{D}}

\newcommand{\algA}{\mathcal{A}}    

\newcommand{\Sym}{{\mathcal{S}} }  
\newcommand{\M}{\mathcal{M}}  
\newcommand{\N}{\mathcal{N}} 

\begin{document}
\renewcommand{\refname}{References}
\thispagestyle{empty}

\title{Free relations for matrix invariants in modular case}%
\author{{A.A. Lopatin}}%
\noindent\address{\noindent{}Artem A. Lopatin%
\newline\hphantom{iiii} Omsk Branch of
\newline\hphantom{iiii} Sobolev Institute of Mathematics, SB RAS,
\newline\hphantom{iiii} Pevtsova street, 13,
\newline\hphantom{iiii} 644099, Omsk, Russia%
\newline\hphantom{iiii} http://www.iitam.omsk.net.ru/\~{}lopatin}%
\email{artem\underline{ }lopatin@yahoo.com}%

\vspace{1cm}
\maketitle {\small
\begin{quote}
\noindent{\sc Abstract. } A classical linear group $G<GL(n)$ acts on $d$-tuples of $n\times n$ matrices by simultaneous conjugation. Working over an infinite field of characteristic different from two we establish that the ideal of free relations, i.e. relations valid for matrices of any order, between generators for matrix $O(n)$-~and $\Sp(n)$-invariants is zero. We also prove similar result for invariants of mixed representations of quivers. 

These results can be considered as a generalization of the characteristic isomorphism ${\rm ch}:\Sym\to J$ between the graded ring $\Sym=\otimes_{d=0}^{\infty} \Sym_d$, where $\Sym_d$ is the character group of the symmetric group $S_d$, and the inverse limit $J$ with respect to $n$ of rings of symmetric polynomials in $n$ variables.  

As a consequence, we complete the description of relations between generators for  $O(n)$-invariants as well as the description of relations for invariants of mixed representations of quivers. We also obtain an independent proof of the result that the ideal of free relations for $GL(n)$-invariants is zero, which was proved by Donkin in [Math. Proc. Cambridge Philos. Soc. 113 (1993), 23--43].

\medskip


\end{quote}
}

\section{Introduction}\label{section_intro}

We assume that $\FF$ is an infinite field of arbitrary characteristic $p=\Char{\FF}$. All vector spaces, algebras and modules are over $\FF$ and all algebras are associative with unity unless otherwise stated.

\subsection{Matrix invariants}
 
Consider a group $G$ from the list $GL(n)$, $O(n)=\{A\in \FF^{n\times n}\,|\,A A^T =E\}$, $\Sp(n)=\{A\in \FF^{n\times n}\,|\,AA^{\ast}=E\}$,  where we assume that $p\neq2$ in case $G$ is $O(n)$ and $n$ is even in the case of $\Sp(n)$. Here $\FF^{n\times n}$ is the space of $n\times n$ matrices over $\FF$ and $A^{\ast}=-J A^T J$ is the symplectic transpose of $A$, where   %
$J=\left(
\begin{array}{cc}
0& E \\
-E& 0\\
\end{array}
\right)$ is the matrix of the skew-symmetric bilinear form. The group $G$ acts on $V=(\FF^{n\times n})^{\oplus d}$ by the diagonal conjugation: 
$$g\cdot (A_1,\ldots,A_d)=(gA_1g^{-1},\ldots,gA_dg^{-1})$$
for $g\in G$ and $A_1,\ldots,A_d$ in $\FF^{n\times n}$.
The coordinate algebra of $V$ is the polynomial ring 
$$R=\FF[x_{ij}(k)\,|\,1\leq i,j\leq n,\, 1\leq k\leq d]$$ 
in $n^2 d$ variables. The ring $R$ is generated by the entries of {\it generic matrices} $X_k=(x_{ij}(k))_{1\leq i,j\leq n}$ ($1\leq k\leq d$). The action of $G$ on $V$ induces the action on $R$ as follows:
$$g\cdot x_{ij}(k)=(i,j)^{{\rm th}}\;{\rm entry}\; {\rm of }\;g^{-1}X_k\,g.$$%
Denote by $R^G$ the algebra of {\it matrix $G$-invariants}, where
$$f\in R^G \;{\rm iff }\;g\cdot f=f \;{\rm for\; all }\; g\in G.$$
Consider an arbitrary $n\times n$ matrix $X$. Denote coefficients in the characteristic polynomial
of $X$ by $\sigma_t(X)$, i.e., %
$$\det(X+\lambda E)=\sum_{t=0}^{n} \lambda^{n-t}\sigma_t(X).$$
So, $\sigma_0(X)=1$, $\sigma_1(X)=\tr(X)$ and $\sigma_n(X)=\det(X)$. 

The following definitions were given in~\cite{Lopatin_Orel}. Let $\M$ be the monoid (without unity) freely generated by {\it letters}  
\begin{enumerate}
\item[$\bullet$] $x_1,\ldots,x_d$, if $G=GL(n)$;

\item[$\bullet$] $x_1,\ldots,x_d,x_1^T,\ldots,x_d^T$, otherwise.
\end{enumerate} 
Assume that $a=a_1\cdots a_r$ and $b$ are elements of $\M$, where $a_1,\ldots,a_r$ are letters. 
\begin{enumerate}
\item[$\bullet$] Introduce an involution ${}^T$ on $\M$ as follows. If $G=GL(n)$, then $a^T=a$. Otherwise, define $b^{TT}=b$ for a letter $b$ and $a^T=a_r^T\cdots a_1^T\in\M$. 

\item[$\bullet$] We say that $a$ and $b$ are {\it cyclic equivalent} and write $a\stackrel{c}{\sim} b$
if there exists a cyclic permutation $\pi\in S_r$ such that
$a_{\pi(1)}\cdots a_{\pi(r)}=b$.  If $a\stackrel{c}{\sim} b$ or $a\stackrel{c}{\sim} b^T$, then we say that $a$ and $b$ are {\it equivalent} and write $a\sim b$.
\end{enumerate}
An element from $\M$ is called {\it primitive} if it is not equal to the power of a shorter monomial. 
\begin{enumerate}
\item[$\bullet$] Let $\N\subset\M$  be the subset of primitive elements.

\item[$\bullet$] Let $\N_{\si}$ be a ring with unity of (commutative) polynomials over $\FF$ freely generated by ``symbolic'' elements $\si_t(a)$, where $t>0$ and $a\in\N$ ranges over $\sim$-equivalence classes. 
\end{enumerate}
We will use the following conventions: $\si_0(a)=1$ and $\si_1(a)=\tr(a)$, where $a\in\N$.  For a letter $b\in\M$ define
$$X_{b}=
\left\{
\begin{array}{rl}
X_{k},&\text{if } b=x_k\\
X_{k}^T,&\text{if } b=x_k^T \text{ and }G=O(n)\\
X_{k}^{\ast},&\text{if } b=x_k^T \text{ and }G=\Sp(n)\\
\end{array}
\right..
$$
Given $a=a_1\cdots a_r\in\M$, where $a_i$ is a letter, we set $X_{a}=X_{a_1}\cdots X_{a_r}$. 
It is known that the algebra of matrix $G$-invariants $R^G\subset R$ is generated over $\FF$ by $\sigma_t(X_a)$, where $1\leq t\leq n$ and $a\in\N$.  These results were established in~\cite{Sibirskii_1968},~\cite{Procesi76} in characteristic zero case and in~\cite{Donkin92a},~\cite{Zubkov99} in the general case. Note that if in the case of $p=0$ we drop the restriction that $a$ is primitive, then it is enough to take $\tr(X_a)$ instead of $\sigma_t(X_a)$, $1\leq t\leq n$, in the description of generators for $R^G$. Relations between the mentioned generators were established by Razmyslov~\cite{Razmyslov74},~Procesi~\cite{Procesi76} in case $p=0$ and Zubkov~\cite{Zubkov96} in case $G=GL(n)$ and $p>0$. 

Consider a surjective homomorphism 
$$\Psi_n:\N_{\si} \to R^G$$ %
defined by $\si_t(a) \to \si_t(X_a)$, if $t\leq n$, and $\si_t(a) \to 0$ otherwise. Note that for any $n\times n$ matrices $A,B$ over $R$ and $1\leq t\leq n$ we have $\si_t(A^{\de})=\si_t(A)$, $(A^{\de})^{\de}=A$, and $(AB)^{\de}=B^{\de}A^{\de}$, where $\de$ stands for the transposition or symplectic transposition. Hence, the map $\Psi_n$ is well defined. Its kernel $K_{n}$ is the ideal of {\it relations} for $R^G$. Elements of $$K_{\infty}=\bigcap_{i>0} K_{i}$$ 
are called {\it free} relations. In other words, a relation between the above mentioned generators for $R^G$ is called free if it is valid for $n\times n$ generic matrices for an arbitrary $n>0$.  In characteristic zero case all free relations are zero (for example, see~\cite{Procesi76}). We generalize this result to the case of arbitrary characteristic different from two: 

\begin{theo}\label{theo_free} If $G$ is $O(n)$ or $\Sp(n)$ and $p\neq2$, then the ideal $K_{\infty}$ of free relations for $R^{G}$ is zero.
\end{theo}
\bigskip%

\noindent{}This theorem is proven at the end of Section~\ref{section_O_lin}. As a consequence, we obtain an independent proof of the result by Donkin~\cite{Donkin93a} that for an arbitrary $p$ there is no free relations for $R^{GL(n)}$ (see Remark~\ref{remark_Donkin}).  The following conjecture is discussed in Remark~\ref{remark_conj}.

\begin{conj}\label{conj1} If $p=2$, then the ideal $K_{\infty}$ of free relations for $R^{\Sp(n)}$ is generated by $\si_t(a)$ for $a\in\N$ satisfying $a\stackrel{c}{\sim}a^T$ and odd $t>0$.
\end{conj}
\bigskip

Applying Theorem~\ref{theo_free} to Theorem~1.1 from~\cite{Lopatin_Orel}, which was proved using an approach from~\cite{ZubkovII}, we complete the description of relations between generators for $R^{O(n)}$:

\begin{theo}\label{theo_relations} If $G=O(n)$, then the ideal of relations $K_{n}$ for $R^{O(n)}\simeq  \N_{\si}/K_{n}$ is generated by $\sigma_{t,r}(a,b,c)$, where $t+2r>n$ ($t,r\geq0$) and $a,b,c$ are linear combinations of elements from $\M$.
\end{theo}
\bigskip%

\noindent{}The exact definition of $\sigma_{t,r}(a,b,c)\in\N_{\si}$ can be found in Section~3 of~\cite{Lopatin_Orel}. Note that the function $\sigma_{t,r}$ was introduced by Zubkov in~\cite{ZubkovII} and it relates to the {\it determinant-pfaffian} from~\cite{LZ1} in the same way as $\sigma_t$ relates to the determinant. More details on the different approaches to the definition of $\sigma_{t,r}$ can be found in Section~1.3 of~\cite{Lopatin_Orel}. Let us remark that the definition of $\sigma_{t,r}$ from~\cite{Lopatin_Orel} is slightly different from the original definition from~\cite{ZubkovII} (see Lemma~7.14 of~\cite{Lopatin_Orel} for details). 

The description of relations for $R^{O(n)}$ was applied to the special case of $n=3$  in~\cite{Lopatin_O3} and~\cite{Lopatin_skewO3}.


\subsection{Mixed representations of quivers}

The notion of {\it supermixed} representations of a quiver was introduced by Zubkov~\cite{Zubkov_preprint} and~\cite{ZubkovI}. It is equivalent to the notion of representations of a {\it signed} quiver considered by Shmelkin in~\cite{Shmelkin}. {\it Orthogonal} and {\it symplectic} representations of {\it symmetric} quivers studied by Derksen and Weyman in~\cite{DW02} as well as {\it mixed} representations of quivers are partial cases of supermixed representations. More details can be found in Section~2.1 of~\cite{Lopatin_so_inv}.  

Section~\ref{section_mixed} is dedicated to the algebra of invariants $I(\Q,\n,\i)$ of mixed representations of a quiver $\Q$. As example, a special case of $I(\Q,\n,\i)$ is the algebra of invariants of bilinear forms on vector spaces $V=\FF^n$ and $V^{\ast}$ under the action of $GL(n)$ as base change (see Example~\ref{ex_forms}).   

Zubkov established generators for $I(\Q,\n,\i)$ in~\cite{ZubkovI} and  
described relations between generators modulo free relations in~\cite{ZubkovII}. In Lemma~\ref{lemma_free_mixed} we show that there are no non-zero free relations for $I(\Q,\n,\i)$ and in Theorem~\ref{theo_mixed} we complete the description of relations for $I(\Q,\n,\i)$. Applying Theorem~\ref{theo_mixed}, in upcoming paper~\cite{Lopatin_form2} we will explicitly describe a minimal generating set for invariants of bilinear forms in dimension two case.

\section{Auxiliaries}\label{section_auxiliary}

Denote the degree of $a\in\M$ by $\deg{a}$, the degree of $a$ in a letter $b$ (i.e., the number of appearances of the letters $b$ and $b^T$ in $a$) by $\deg_{b}{a}$. For $t>0$ we set $\deg{\si_t(a)}=t\deg{a}$ and define $\deg{f}$ for a monomial $f\in\N_{\si}$ in the natural way. 

Assume that $m$ is a positive integer. Given $i\in\ZZ$, we write $|i|_m$ for $1\leq j\leq m$  such that $i\equiv j\,({\rm mod}\; m)$. 

\definition{of $l$-subword} Assume that $a=a_1\cdots a_r$, $b=b_1\cdots b_s\in\M$, and $l>0$, where $a_i,b_j$ are letters for all $i,j$. We say that
\begin{enumerate}
\item[$\bullet$] $a$ is an {\it $l$-subword} of $b$, if $a_i=b_{|l+i-1|_s}$ for all $1\leq i\leq r$; 
\item[$\bullet$] $a$ is an {\it $l^T$-subword} of $b$, if $a_i=b_{|l-i+1|_s}^T$ for all $1\leq i\leq r$.
\end{enumerate}
\bigskip
 
As an example, for $a=x_1x_2x_3^Tx_4\in\M$ we have that $x_3^Tx_4x_1$ is a $3$-subword of $a$ and $x_1^Tx_4^T$ is an $1^T$-subword of $a$.

\begin{remark}\label{remark_subword} If $a\sim b$ for $a,b\in\M$, then $a$ is an $l$-subword of $b$ or $l^T$-subword of $b$ for some $l>0$. 
\end{remark}

\begin{lemma}\label{lemma_subword00} Let $b,c\in\M$. Then
\begin{enumerate}
\item[a)] if $bc=cb$, then there is an $e\in\M$ such that $b=e^i$ and $c=e^j$ for some $i,j>0$. 

\item[b)] if $b=b^T$, then $b=c c^T$ for a $c\in\M$.
\end{enumerate}
\end{lemma}
\begin{proof}
a) The proof is by induction on $\deg{b}+\deg{c}>1$. If $\deg{b}+\deg{c}=2$, then $b=c$ is a letter and the statement is valid.

If $\deg{b}=\deg{c}$, then $b=c$ and the required is proven. Otherwise, without loss of generality we can assume that $\deg{b}>\deg{c}$. Then $b=c b_1$ for a $b_1\in\M$. Thus, $b_1 c=c b_1$. Induction hypothesis completes the proof. 

b) The proof is by induction on $\deg{b}>0$. If $b$ is a letter, then $b\neq b^T$. Otherwise, 
$b=y b_1$ for a letter $y$ and $b_1\in\M$. Since $y b_1 = b_1^T y^T$, we have $b_1 = b_2 y^T$ for a $b_2\in\M_1$. Then $b_2=b_2^T$ and the induction hypothesis implies $b_2=c c^T$ for a $c\in\M_1$. Therefore, $b=y c (y c)^T$.  
\end{proof}

\begin{lemma}\label{lemma_subword} Assume that $a\in\N$ with $\deg{a}=r$ and $1\leq l\leq r$. Then 
\begin{enumerate}
\item[a)] $a$ is an $l$-subword of $a$ if and only if $l=1$;

\item[b)] if $a\stackrel{c}{\not\sim}a^T$, then $a$ is not an $l^T$-subword of $a$;

\item[c)] $a\stackrel{c}{\sim}a^T$ if and only if there exists a $b\in\N$ satisfying $a\stackrel{c}{\sim}b$ and $b=b^T$; 

\item[d)] if $a=a^T$, then $a$ is an $l^T$-subword of $a$ if and only if $l=r$. 
\end{enumerate}
\end{lemma}
\begin{proof}  a) If $a$ is an $l$-subword of $a$ for $1< l\leq m$, then there are $a_1,a_2\in\M$ such that $a=a_1 a_2$ and $a_1a_2=a_2a_1$. Part~a) of Lemma~\ref{lemma_subword00} implies a contradiction.

b) If $a$ is an $l^T$-subword of $a$, then $a=a_1a_2$ for $a_1,a_2\in\M_1$ satisfying $a_1a_2=a_1^T a_2^T$. Thus $a_i=a_i^T$ and, by part~b) of Lemma~\ref{lemma_subword00}, $a_i=c_i c_i^T$ for $c_1,c_2\in\M_1$ ($i=1,2$). Therefore, $a=c_1c_1^T c_2c_2^T\stackrel{c}{\sim} a^T$; a contradiction.

c) Since $a\stackrel{c}{\sim} a^T$, we have $a=a_1a_2$ for $a_1,a_2\in\M_1$ satisfying $a_1a_2=a_1^T a_2^T$. As in the proof of part~b), we obtain $a=c_1c_1^T c_2c_2^T$, where $c_1,c_2\in\M_1$. Then $b=c_1^Tc_2 (c_1^T c_2)^T$ satisfies the required condition.

d) Let $a=a^T$ and $a$ be an $l^T$-subword of $a$ for $1\leq l< m$. As in the proof of part~b), we obtain $a=c_1c_1^T c_2c_2^T$, where $c_1,c_2\in\M$. Then for $b_i=c_ic_i^T$ ($i=1,2$) we have $b_1b_2=b_2b_1$. Part~a) of Lemma~\ref{lemma_subword00} implies that $a=b_1 b_2$ is not primitive; a contradiction.
\end{proof}

\section{Derivations}\label{section_O_derivations}

In this section we assume that $G$ is $O(n)$ or $\Sp(n)$. Given $q>0$, we set
$$\widehat{R}=R\otimes \FF[y_{ij}(k,q)\,|\,1\leq i,j\leq n,1\leq k\leq d, q>0]\;\text{ and } \;Y_{k,q}=(y_{ij}(k,q))_{1\leq i,j\leq n}.$$
Let $G$ act on $\widehat{R}$ by the same way as on $R$:
$$g\cdot y_{ij}(k,q)=(i,j)^{{\rm th}}\;{\rm entry}\; {\rm of }\;g^{-1}Y_{k,q}\,g.$$%
Define a linear map $\partial_q:\widehat{R}\to \widehat{R}$ as follows: given an $f\in\widehat{R}$, we have
$$\partial_q(f)=\sum_{k=1}^d\sum_{1\leq i,j\leq n}\frac{\partial{f}}{\partial{x_{ij}(k)}}\, y_{ij}(k,q),$$
where $\frac{\partial{f}}{\partial{x_{ij}(k)}}$ stands for the partial derivation. As an example, if $f=x_{11}(1)\, y_{22}(2,r)$, then
$$\partial_q(f)=y_{11}(1,q)\, y_{22}(2,r) \text{ and }\partial_2\partial_1(f)=0.$$

\noindent{}For an $n\times n$ matrix $A=(f_{ij})_{1\leq i,j\leq n}$ over $\widehat{R}$ we set $\partial_q(A)=(\partial_q(f_{ij}))_{1\leq i,j\leq n}$.
Obviously, $\partial_q(X_k)=Y_{k,q}$, $\partial_q(Y_{k,r})=0$  and $\partial_q(A^T)=\partial_q(A)^T$ for $q,r>0$.

\remark{}\label{remark1} The linear map $\partial_q$ has the usual properties of the   derivation. Namely, if $f,h\in\widehat{R}$ and $m>0$, then 
$$\partial_q(fh)=\partial_q(f)h+f\partial_q(h)\text{ and } \partial_q(f^m)=m f^{m-1}\partial_q(f).$$

\begin{lemma}\label{lemma_properties}
For any $n\times n$ matrices $A,B$ over $\widehat{R}$ and $q>0$ the following properties hold: 
\begin{enumerate}
\item[a)] $\partial_q(AB)=\partial_q(A)B + A\,\partial_q(B)$;

\item[b)] $\partial_q(\si_t(A))=\sum_{i=0}^{t-1} (-1)^i \tr(A^i \partial_q(A)) \,\si_{t-i-1}(A)$  for all $1\leq t\leq n$. 
\end{enumerate}
\end{lemma}
\begin{proof}For short, we write $\partial$ for $\partial_q$. Let $A=(f_{ij})_{1\leq i,j\leq n}$ and $B=(h_{ij})_{1\leq i,j\leq n}$. Remark~\ref{remark1} implies that $(i,j)^{\rm th}$ entry of $\partial(AB)$ is equal to 
$$\sum_{k=1}^n\left( \partial(f_{ik})h_{kj} + f_{ik}\partial(h_{kj})\right).$$
Hence, part~a) is proven. Since
\begin{eq}\label{eq_4} 
\si_t(A)=\sum_{1\leq i_1<\cdots <i_t\leq n}\;\sum_{\tau\in S_t}\sign(\tau)\, f_{i_1,i_{\tau(1)}}\cdots f_{i_t,i_{\tau(t)}},
\end{eq}%
Remark~\ref{remark1} implies that $\partial(\si_t(A))$ is the coefficient of $\lambda^{t-1}\mu$ in the polynomial $\si_t(\la A + \mu \partial(A))$ in $\la,\mu$. Amitsur's formula from~\cite{Amitsur_1980} completes the proof of part~b).  
\end{proof}

\example{}\label{ex1} Applying Lemma~\ref{lemma_properties}, we obtain the next equalities. 
\begin{enumerate}
\item[$\bullet$] Let $f=\si_2(X_1)$. Then $\partial_q(f)=-\tr(Y_{1q}X_1) + \tr(X_1)\tr(Y_{1q})$.

\item[$\bullet$] Let $f=\tr(X_1)^2\tr(X_1 X_2)$ and $p=2$. Then $\partial_1(f)=\tr(X_1)^2 \tr(Y_{11} X_2) + \tr(X_1)^2\tr(X_1 Y_{21})$ and $\partial_2\partial_1(f)=\tr(X_1)^2 \tr(Y_{11} Y_{22}) + \tr(X_1)^2\tr(Y_{12} Y_{21})$. Note that $\partial_1\partial_1(f)=0$. 
\end{enumerate}
\bigskip

Similarly to $\M,\N,\N_{\si}$, we introduce the following notions.
\begin{enumerate}
\item[$\bullet$] For $q\geq0$ we denote by $\widehat{\M}(q)$ the monoid (without unity) freely generated by {\it letters}  $x_k,x_k^T,y_{ks},y_{ks}^T$, where $1\leq k\leq d$ and $1\leq s\leq q$, and set $\widehat{\M}=\cup_{q>0} \widehat{\M}(q)$. Note that $\widehat{\M}(0)=\M$.

\item[$\bullet$] Define the involution ${}^T$ and the equivalences $\sim$ and $\stackrel{c}{\sim}$ on $\widehat{\M}$ in the same way as they were defined on $\M$.

\item[$\bullet$] Let $\widehat{\M}_1=\widehat{\M}\sqcup\{1\}$ and $\widehat{\M}_{\FF}$ be the vector space with the basis $\widehat{\M}$.

\item[$\bullet$] We denote by $\widehat{\N}(q)\subset\widehat{\M}(q)$ and $\widehat{\N}\subset\widehat{\M}$  subsets of primitive elements and define $\widehat{\N}_{\si}(q)$, $\widehat{\N}_{\si}$ similarly to $\N_{\si}$. 
\end{enumerate} 
Define the notion of degree for $\widehat{\M}$ and $\widehat{\N}_{\si}$ in the same way as for $\M$ and $\N_{\si}$. For a letter $b\in\widehat{\M}$ define the $n\times n$ matrix $X_{b}$ as follows:
$$X_{b}=
\left\{
\begin{array}{rl}
Y_{ks},&\text{if } b=y_{ks}\\
Y_{ks}^T,&\text{if } b=y_{ks}^T \text{ and }G=O(n)\\
Y_{ks}^{\ast},&\text{if } b=y_{ks}^T \text{ and }G=\Sp(n)\\
X_b,&\text{if } b\in\M\\
\end{array}
\right..
$$
Given letters $a_1,\ldots, a_r\in\widehat{\M}$, we set $X_{a_1\cdots a_r}=X_{a_1}\cdots X_{a_r}$. A homomorphism  $\widehat{\Psi}_n:\N_{\si} \to R^G$ is defined by $\si_t(a) \to \si_t(X_a)$, if $t\leq n$, and $\si_t(a) \to 0$ otherwise.  

Assume that $q>0$. We define a linear map $\partial_q:\widehat{\M}_{\FF}\to \widehat{\M}_{\FF}$ as follows:  
\begin{enumerate}
\item[$\bullet$] $\partial_q(x_k)=y_{kq}$, $\partial_q(x_k^T)=y_{kq}^T$, and $\partial_q(y_{ks})=\partial_q(y_{ks}^T)=0$ for all $1\leq k\leq d$ and $s>0$;

\item[$\bullet$] $\partial_q(a_1\cdots a_r)=\sum_{i=1}^r a_1\cdots a_{i-1}\, \partial_q(a_i)\, a_{i+1} \cdots a_r$ for letters $a_1,\ldots, a_r\in\widehat{\M}$.
\end{enumerate}

Define a linear map $\partial_q:\widehat{\N}_{\si}(q-1)\to \widehat{\N}_{\si}(q)$ as follows: for $a,a_1,\ldots, a_r\in\widehat{\M}(q-1)$ and $t,t_1,\ldots,t_r>0$, we set 
\begin{enumerate}
\item[$\bullet$] $\partial_q(\al)=0$ for $\al\in\FF$;

\item[$\bullet$] $\partial_q(\si_t(a))=\sum_{i=0}^{t-1} (-1)^i \tr(a^i \partial_q(a)) \,\si_{t-i-1}(a)$, where we use the convention that 
$$\tr(a_1+\cdots+a_r)=\tr(a_1)+\cdots+\tr(a_r);$$

\item[$\bullet$] $\partial_q(\si_{t_1}(a_1)\cdots \si_{t_r}(a_r))=$ 
$$\sum_{i=1}^r \si_{t_1}(a_1)\cdots \si_{t_{i-1}}(a_{i-1})\, \partial_q(\si_{t_i}(a_i))\, \si_{t_{i+1}}(a_{i+1}) \cdots \si_{t_r}(a_r).$$
\end{enumerate}
Since an element $(a_1\cdots a_r)^i a_1\cdots a_{j-1}\, \partial_q(a_j)\, a_{j+1} \cdots a_r$ is either primitive or zero, where $a_1,\ldots,a_r\in\widehat{\M}(q-1)$ are letters, the following remark implies that the map $\partial_q:\widehat{\N}_{\si}(q-1)\to \widehat{\N}_{\si}(q)$ is well defined. 

\begin{remark}\label{remark_M} For $a,b\in \widehat{\M}$ and $q>0$ we have
\begin{enumerate}
\item[$\bullet$] $\partial_q(a^T)=\partial_q(a)^T$;

\item[$\bullet$] if $a\sim b$, then $\tr(a^m\partial_q(a))=\tr(b^m\partial_q(b))$ for all $m>0$.
\end{enumerate}
\end{remark}
\bigskip

Note that by abuse of notation we denote three different linear maps by one and the same symbol $\partial_q$. 

\begin{lemma}\label{lemma_diagram}
For $q>0$ the following diagram is commutative:
$$\begin{array}{ccc}
\widehat{\N}_{\si}(q-1)&\stackrel{\widehat{\Psi}_n}{\longrightarrow}& \widehat{R}^G\\
\partial_q\downarrow\quad\;\; & &\quad\;\;\downarrow\partial_q\\
\widehat{\N}_{\si}(q)&\stackrel{\widehat{\Psi}_n}{\longrightarrow}& \widehat{R}^G\\
\end{array}$$
In particular, if $f\in\widehat{\N}_{\si}(q-1)$ is a free relation (i.e., $\widehat{\Psi}_n(f)=0$ for all $n>0$), then $\partial_q(f)\in\widehat{\N}_{\si}$ is also a free relation.
\end{lemma}
\begin{proof}
Obviously, the statement of the lemma holds for $\tr(X_a)$, where $a\in\widehat{\M}$ is a letter. Lemma~\ref{lemma_properties} and the definition of $\partial_q:\widehat{\N}_{\si}(q-1)\to \widehat{\N}_{\si}(q)$ complete the proof.
\end{proof}

In the proof of the next lemma we use statements from Section~\ref{section_auxiliary}, which obviously hold for elements from $\widehat{\M}$. 

\begin{lemma}\label{lemma_linind1}
Assume that $q,s>0$ and $a\in\widehat{\N}(q-1)$ satisfies 
\begin{enumerate}
\item[a)] $\deg_{x_k}(a)\neq0$ for some $1\leq k\leq d$; 

\item[b)] $a\stackrel{c}{\not\sim} a^T$. 
\end{enumerate}
Let monomials 
$$f_i=\prod_{j=1}^{r_i} \si_{t_{ij}}^{m_{ij}}(a)\in\widehat{\N}_{\si},$$ 
be pairwise different, where $r_i>0$, $t_{i1}>\cdots>t_{ir_i}\geq1$ and $m_{i1},\ldots,m_{ir_i}>0$ are not divided by $p$ ($1\leq i\leq s$). Then 
$\partial_q(f_1),\ldots,\partial_q(f_s)$ are linear independent over $\FF$. 
\end{lemma}
\begin{proof}  Let $\al_1\partial_q(f_1)+\cdots+\al_s\partial_q(f_s)=0$, where $\al_i\in\FF$, be a non-trivial linear combination. Then without loss of generality we can assume that $\al_i\neq0$ for all $i$. Moreover, without loss of generality we can assume that 
$$t_{11}=\max_{i,j}\{t_{ij}\}.$$
For short, we denote $t=t_{11}$. We have 
$$\partial_q(a)= b_1+\cdots+b_r$$
for pairwise different $b_1,\ldots,b_r\in\widehat{\N}$ and $r>0$. By the definition of $\partial_q$,  
$$
\partial_q(f_i)=\sum_{w\in\Omega_i} \be_{i,w} f_{i,w},
$$
where $\Omega_i$ is equal to the set of pairs 
$$\{(u,v,k)\,|\,1 \leq u\leq r_i,\,0\leq v< t_{iu},\,1\leq k\leq r\},$$ 
$\be_{i,(u,v,k)}=(-1)^v m_{iu}$ is non-zero, and
$$f_{i,(u,v,k)}=\si_{t_{iu}}(a)^{m_{iu}-1} \tr(a^v b_k) \,\si_{t_{iu}-v-1}(a) \prod_{1\leq j\leq r_i,\,j\neq u} \si_{t_{ij}}^{m_{ij}}(a).$$
We claim that for $i_0=1$ and $w_0=(1,t-1,1)$ the following statement holds:

\medskip%
\noindent{}{\it If $f_{i,w}=f_{i_0,w_0}$, then $i=i_0$ and $w=w_0$.}
\medskip

Assume that $i$ and $w=(u,v,k)$ satisfy $f_{i,w}=f_{i_0,w_0}$. There exists a unique  $1\leq k'\leq d$ such that $\deg_{z}(b_1)=1$ for $z=y_{k',q}$. The only multiplier of $f_{i_0,w_0}$  that contains $z$ or $z^T$ is $\tr(a^{t-1} b_1)$ and the only multiplier of $f_{i,w}$ that can contain $z$ or $z^T$ is $\tr(a^v b_k)$. Therefore, $a^{t-1} b_1\sim a^v b_k$. The last equivalence implies that $v=t-1$ and one of the following cases holds:
\begin{enumerate}
\item[1.] $a^{t-1} b_1\stackrel{c}{\sim} a^{t-1} b_k$; 

\item[2.] $a^{t-1} b_1\stackrel{c}{\sim} (a^{t-1} b_k)^T$. 
\end{enumerate}
Note that the result of the substitutions $z\to x_{k'}$, $z^T\to x_{k'}^T$ in $b_{1}$ as well as in $b_{k}$ is $a$. Thus, making these substitutions in the above equivalences, we obtain that $a$ is the $l$-subword of $a$ in case~1 and $a$ is the $l^T$-subword of $a$ in case~2, where $1\leq l\leq \deg{a}$. Part~b) of Lemma~\ref{lemma_subword} implies a contradiction in case~2. Part~a) of Lemma~\ref{lemma_subword} implies that $l=1$ in case~1. Thus, $a^{t-1} b_1= a^{t-1} b_k$ and $k=1$. Since $v< t_{iu}\leq t$, we obtain $t_{iu}=t$. The inequalities $t_{i,1}>\cdots>t_{i,r_i}$ imply $u=1$. Therefore, $w=w_0$. 

It is not difficult to see that in the quotient field of $\widehat{\N}_{\si}$ we have
$$f_{i,w}=\frac{f_i}{\si_t(a)}\tr(a^{t-1} b_1)\text{ and }
f_{i_0,w_0}=\frac{f_1}{\si_t(a)}\tr(a^{t-1} b_1).$$
Thus $i=1$ and the claim is proven. Obviously, the claim implies a contradiction to the fact that $\al_{i_0}\neq0$. The lemma is proven. 
\end{proof}
\bigskip

\begin{lemma}\label{lemma_linind2}
Assume that in the formulation of Lemma~\ref{lemma_linind1} we have $a\stackrel{c}{\sim} a^T$ instead of condition~b). Then 
\begin{enumerate}
\item[$\bullet$] if $p\neq 2$, then $\partial_q(f_1) ,\ldots, \partial_q(f_s)$ are linear independent over $\FF$;

\item[$\bullet$] if $p=2$, then $\partial_q(f_i)=0$ for all $i$.
\end{enumerate}
\end{lemma}
\begin{proof} We use notations from the formulation of Lemma~\ref{lemma_linind1}.  Without loss of generality we can assume that 
$$t_{11}=\max_{i,j}\{t_{ij}\}.$$
For short, we denote $t=t_{11}$. By part~b) of Lemma~\ref{lemma_subword00} and part~c) of Lemma~\ref{lemma_subword}, without loss of generality we can assume that $a=cc^T$ for some $c\in\widehat{\M}$.  We have $\partial_q(c)= b_1+\cdots+b_r$,
where $b_1,\ldots,b_r\in\widehat{\M}$ are pairwise different and $r>0$. By Remark~\ref{remark_M},  
$$\partial_q(a)=b_1 c^T + \cdots + b_r c^T + c b_1^T +\cdots + c b_r^T.$$
Since $\tr(a^v b_k c^T)=\tr(a^v c b_k^T)$ for all $1\leq k\leq r$ and $v>0$, we obtain that 
$$
\partial_q(f_i)=\sum_{w\in\Omega_i} \be_{i,w} f_{i,w},
$$%
where $\Omega_i$ is equal to the set of pairs 
$$\{(u,v,k)\,|\,1 \leq u\leq r_i,\,0\leq v< t_{iu},\,1\leq k\leq r\},$$ 
$\be_{i,(u,v,k)}=2(-1)^v m_{iu}$, and
$$f_{i,(u,v,k)}=\si_{t_{iu}}(a)^{m_{iu}-1} \tr(a^v b_k c^T) \,\si_{t_{iu}-v-1}(a) \prod_{1\leq j\leq r_i,\,j\neq u} \si_{t_{ij}}^{m_{ij}}(a).$$
Therefore, if $p=2$, then $\partial_q(f_i)=0$ for all $i$ and the required is proven. 

Let $p\neq 2$. We claim that if $f_{i,w}=f_{i_0,w_0}$, then $i=i_0$ and $w=w_0$, where  $i_0=1$ and $w_0=(1,t-1,1)$. 

Let $i$ and $w=(u,v,k)$ satisfy $f_{i,w}=f_{i_0,w_0}$. 
There exists a unique  $1\leq k'\leq d$ such that $\deg_{z}(b_1)=1$ for $z=y_{k',q}$.
The only multiplier of $f_{i_0,w_0}$  that contains $z$ or $z^T$ is $\tr(a^{t-1} b_1 c^T)$ and the only multiplier of $f_{i,w}$ that can contain $z$ or $z^T$ is $\tr(a^v b_k c^T)$. Therefore, $a^{t-1} b_1 c^T\sim a^v b_k c^T$. The last equivalence implies that $v=t-1$ and one of the following cases holds:
\begin{enumerate}
\item[1.] $a^{t-1} b_1 c^T\stackrel{c}{\sim} a^{t-1} b_k c^T$; 

\item[2.] $a^{t-1} b_1 c^T\stackrel{c}{\sim} a^{t-1} c b_k^T$. 
\end{enumerate}
Note that the result of the substitutions $z\to x_{k'}$, $z^T\to x_{k'}^T$ in $b_{1}$ as well as in $b_{k}$ is $c$. Making these substitutions, we obtain that $a$ is the $l$-subword of $a$ in both cases, where $1\leq l\leq \deg{a}$. Part~a) of Lemma~\ref{lemma_subword} implies that $l=1$. Since $\deg_z(a)=0$, we obtain a contradiction in case~2 and the equality $a^{t-1} b_1= a^{t-1} b_k$ in case~1.

%

So, we proved that $k=1$. The rest of the proof of the claim is the same as in the proof of Lemma~\ref{lemma_linind1} and the required follows from the claim. 
\end{proof}
\bigskip

\section{$p$-multilinear free relations}\label{section_O_reduction}

We assume that $G$ is $O(n)$ or $\Sp(n)$. Let $f\in\widehat{\N}_{\si}$ be a monomial. If $p>0$, then we write  $f=f^{+}f^{-}$ for 
\begin{eq}\label{eq_2}
f^{+}=\si_{t_1}^p(a_1)\cdots \si_{t_r}^p(a_r)\text{ and } 
f^{-}=\si_{l_1}^{q_1}(b_1)\cdots \si_{l_s}^{q_s}(b_s),
\end{eq}%
where $a_1,\ldots,a_r,b_1,\ldots,b_s\in\widehat{\N}$, $1\leq q_1,\ldots,q_s<p$, and $\si_{l_1}(b_1),\ldots, \si_{l_s}(b_s)$ are pairwise different elements of $\widehat{\N}_{\si}$. If $p=0$, then we set $f^{+}=1$ and $f^{-}=f$. 

As example, if $f=\tr^5(x_1)$ and $p=2$, then $f^{+}=\tr^4(x_1)$ and $f^{-}=\tr(x_1)$.

\defin{}\label{def_p_multilin} Let $f=\sum_{w\in\Omega} \al_w f_w\in \widehat{\N}_{\si}$, where $\al_w\in\FF$ is non-zero and $f_w$ is a monomial. Then $f$ is called {\it multilinear} if  $\deg_{z}(f_w)\leq 1$ for every letter $z\in\widehat{\M}$ and $w\in\Omega$.

The element $f$ is called {\it $p$-multilinear} if there is a subset $I\subset\{x_k,y_{k,q}\,|\,1\leq k\leq d,\; q>0\}$ such that every $w\in\Omega$ satisfies the following conditions:    
\begin{enumerate}
\item[$\bullet$] $\deg_{z}(f_{w}^{+})=0$ and $\deg_{z}(f_{w}^{-})\leq1$ for every letter $z\not\in I$;

\item[$\bullet$] $\deg_{z}(f_{w}^{-})=0$ for every $z\in I$.
\end{enumerate}
\bigskip

In this section we prove that if there is a non-zero free relation, then there exists a non-zero $p$-multilinear free relation (see Corollary~\ref{cor_reduction} below). 

For $f$ as in Definition~\ref{def_p_multilin} we set 
$$\degp{f}=\max_{w\in\Omega}\left\{\deg(f_w^{+})\right\}
\;\text{ and }\;
\degm{f}=\max_{w\in\Omega}\left\{\sum_{k=1}^d \deg_{x_{k}}(f_w^{-})\right\}.$$
Note that for $q>0$ and a monomial $f\in\widehat{\N}_{\si}$ we have 
\begin{eq}\label{eq_3}
\partial_q(f)=f^{+}\partial_q(f^{-}).
\end{eq}%

The next remark follows from the definition of $\partial_q$ and the fact that if $b^r=c^s$ for $b,c\in\M$ and $r,s>0$, then there is an $e\in\M$ such that $b=e^i$ and $c=e^j$ for some $i,j>0$ (see part~a) of Lemma~\ref{lemma_subword00}).

\begin{remark}\label{remark_last}
Let $q,t,l>0$, $a,b\in\widehat{\N}(q-1)$, $\partial_q(\si_t(a))=\sum_i\pm f_i$ and $\partial_q(\si_l(b))=\sum_j\pm h_j$ for monomials $f_i,h_j$. Then there exist $i,j$ with $f_i=h_j$ if and only if $a\sim b$ and $t=l$.
\end{remark}

\begin{lemma}\label{lemma_p_multilin}
Let $p\neq 2$ and $f\in\N_{\si}$. Then there is a $q\geq0$ such that
$\partial_q\cdots\partial_1(f)$ is a non-zero $p$-multilinear element of $\widehat{\N}_{\si}$.
\end{lemma}
\begin{proof} Let $q>0$. Consider $h=\sum_{w\in\Omega}\al_w h_w\in \widehat{\N}_{\si}$
for non-zero elements $\al_w\in\FF$ and pairwise different monomials $f_w\in \widehat{\N}_{\si}$. Let $h$ be in $\widehat{\N}_{\si}(q-1)$ and  
\begin{eq}\label{eq_cond}
\deg_{y_{ks}}(h_w^{+})= 0\text{ and } \deg_{y_{ks}}(h_w^{-})\leq 1 
\end{eq}%
for all $1\leq k\leq d$, $s>0$. Then we claim that one the following possibilities holds: 
\begin{enumerate}
\item[$\bullet$] $h$ is a $p$-multilinear; 

\item[$\bullet$] $h'=\partial_q(h)$ is non-zero, $h'$ satisfies condition~\Ref{eq_cond} and $\degm{h'}<\degm{h}$.
\end{enumerate}

If $\degm{h}=0$, then $h$ is $p$-multilinear for $I=\{x_1,\ldots,x_d\}$.

We assume that $\degm{h}>0$. Let $\{a_1,\ldots,a_s\}$ be a subset of $\widehat{\N}$ such that for every $w\in\Omega$ we have
$$h_w=\prod_{1\leq i\leq s} h_{w,i},$$
where $h_{w,i}$ is a product of some elements of the set $\{\si_t(a_i)\,|\, t>0\}$ or $h_{w,i}=1$. Given $1\leq i\leq s$, denote by $\Theta_{i}$ the set of $w\in\Omega$ with $\deg_{x_k}(h^{-}_{w,i})\neq 0$ for some $1\leq k\leq d$. 

Since $\degm{h}>0$, the set $\Theta_{i_0}$ is not empty for some $1\leq i_0\leq s$. Using formula~\Ref{eq_3} and the equality $h_w^{-}=\prod_{i=1}^s h_{w,i}^{-}$, we obtain that for every $w\in\Omega$  
\begin{eq}\label{eq_5}
\partial_q(h_{w})=h_w^{+}\sum h_{w,1}^{-}\cdots h_{w,i-1}^{-}\partial_q(h_{w,i}^{-}) h_{w,i+1}^{-}\cdots h_{w,s}^{-},
\end{eq}%
where the sum ranges over $1\leq i\leq s$ satisfying $w\in\Theta_i$. Note that if  $w\not\in\Theta_i$ for all $i$, then $\partial_q(h_{w})=0$. Applying Lemmas~\ref{lemma_linind1} and~\ref{lemma_linind2}, we obtain that $\{\partial_q(h_{w,i_0}^{-})\,|\,w\in \Theta_{i_0}\}$ are linear independent over $\FF$. Thus, the definition of $h_{w,i}$ together with Remark~\ref{remark_last} implies that $h'=\partial_q(h)\neq0$. It follows from formula~\Ref{eq_5} that $h'$ satisfies  condition~\Ref{eq_cond} and $\degm{h'}<\degm{h}$. Therefore, the claim is proven.

Applying the claim to $f$, $\partial_1(f)$, $\partial_2(\partial_1(f))$ and so on, we prove the required statement by induction on $\degm{f}$.
\end{proof}

\begin{cor}\label{cor_reduction}
Let $G$ be $O(n)$ or $\Sp(n)$ and $p\neq2$. Assume that $f\in \N_{\si}$ is a non-zero free relation and $d>\!\!>\!0$ is large enough. Then there exists a non-zero $p$-multilinear free relation $h\in\N_{\si}$ with $\degp{h}\leq\degp{f}$.
\end{cor}
\begin{proof} 
Applying Lemma~\ref{lemma_p_multilin}, we obtain $q\geq0$ such that $f'=\partial_q\cdots \partial_1(f)$ is a non-zero $p$-multilinear element of $\widehat{\N}_{\si}$. By Lemma~\ref{lemma_diagram}, $f'$ is a free relation. Equality~\Ref{eq_3} implies that  $\degp{f'}\leq\degp{f}$.

Let $f'=\sum\al_i f_i$ for non-zero $\al_i\in\FF$ and pairwise different monomials $f_i$. Since $d$ is large enough, there is an injective map $\varphi$ from the set of $y_{ks}$ satisfying $\deg_{y_{ks}}(f_i)\neq0$ for some $i$ ($1\leq k\leq d$, $s>0$) to the set of $x_j$ satisfying $\deg_{x_j}(f_i)=0$ for all $i$ ($1\leq j\leq d$). Making substitutions $y_{ks}\to\varphi(y_{ks})$ and $y_{ks}^T\to\varphi(y_{ks})^T$ in $f'$, we obtain the required $h\in\N_{\si}$.
\end{proof}

\section{Multilinear free relations}\label{section_O_lin}

We assume that $G$ is $O(n)$ or $\Sp(n)$. Given an $n\times n$ matrix $A=(f_{ij})_{1\leq i,j\leq n}$ over $R$, we denote 
$$A^{(p)}=(f_{ij}^p)_{1\leq i,j\leq n}.$$

\begin{remark}\label{remark_p}
Let $\algA$ be a commutative $\FF$-algebra and $p>0$. Then for $a_1,\ldots,a_r\in\algA$ we have $(a_1+\cdots+a_r)^p=a_1^p+\cdots+a_r^p$. 
\end{remark}

\begin{lemma}\label{lemma_power_p} For $n\times n$ matrices $A$ and $B$ over $R$ the following properties hold:
\begin{enumerate}
\item[a)] $(AB)^{(p)}=A^{(p)}B^{(p)}$;

\item[b)] $\si_t(A)^{p}=\si_t(A^{(p)})$ for $1\leq t\leq n$;

\item[c)] if $n$ is even, then $(A^{\ast})^{(p)}=(A^{(p)})^{\ast}$.
\end{enumerate}
\end{lemma}
\begin{proof}
We set $A=(f_{ij})_{1\leq i,j\leq n}$ and $B=(h_{ij})_{1\leq i,j\leq n}$. Then $(i,j)^{\rm th}$ entry of $AB$ is $(\sum_{k=1}^n f_{ik} h_{kj})^p$ and Remark~\ref{remark_p} completes the proof of part~a). Part~b) follows from formula~\Ref{eq_4} and Remark~\ref{remark_p}. Part~c) follows from part~a).
\end{proof}

\begin{lemma}\label{lemma_multilin}
Assume that $p\neq 2$, $f\in \N_{\si}$ is a non-zero $p$-multilinear free relation and $d>\!\!>\!0$ is large enough. Then there exists a non-zero multilinear free relation in $\N_{\si}$. 
\end{lemma} 
\begin{proof} Without loss of generality we can assume that $p>0$.  
Let $f=\sum_{w\in\Omega}\al_w f_w\in \N_{\si}$ be not multilinear, where $\al_w\in\FF$ is non-zero and $f_w$ is a monomial.  Note that $f_w^{+}=h_w^p$ for some $h_w\in \N_{\si}$. Definition~\ref{def_p_multilin} implies that there is a set $I\subset\{1,\ldots,d\}$ such that for every $w$ the element $f_w^{+}$ ``depends"{} only on $\{x_k\,|\,k\in I\}$ whereas $f_w^{-}$ ``depends"{} only on $\{x_k\,|\,k\not\in I\}$.  Hence $h=\sum_{w\in\Omega}\al_w h_w f_w^{-}$ is a non-zero element of $\N_{\si}$ satisfying $\degp{h}<\degp{f}$. 

Given $n>0$, we have $\Psi_n(f)=0$. By Remark~\ref{remark_p}, $\Psi_n(h_w^{p})$ is a polynomial in $x_{ij}^p(k)$, where $k\in I$ and $1\leq i,j\leq n$. It follows from Lemma~\ref{lemma_power_p} that the result of substitution $x_{ij}^p(k)\to x_{ij}(k)$ ($k\in I$, $1\leq i,j\leq n$) in $\Psi_n(h_w^{p})$ is $\Psi_n(h_w)$. Thus, applying the mentioned substitution to $\Psi_n(f)=0$ we obtain $\Psi_n(h)=0$. Therefore, $h$ is a free relation.

Applying Corollary~\ref{cor_reduction} to $h$, we obtain a non-zero $p$-multilinear free relation $f'$ satisfying $\degp{f'}\leq\degp{h}$. Repeating this procedure several times and using the fact that $\degp{f}$ decreases at each step by at least one, we finally obtain a non-zero multilinear free relation.  
\end{proof}

\begin{lemma}\label{label_no_multilin}
There is no a non-zero multilinear free relation in $\N_{\si}$ for $p\geq0$.
\end{lemma}
\begin{proof} We assume that $f=\sum_{w\in\Omega}\al_w f_w\in \N_{\si}$ is a non-zero multilinear free relation for non-zero $\al_w\in\FF$ and pairwise different monomials $f_w$. Since $\FF$ is infinite, without loss of generality we can assume that $f$ is homogeneous with respect to $\NN^d$-grading of $\N_{\si}$, i.e., $\deg_{x_1}(f_w)=\cdots=\deg_{x_d}(f_w)=1$ for all $w$. 

We set $n=d$ in case $G$ is the orthogonal group and $n=2d$ in case $G$ is the symplectic  group. Denote by $e_{i,j}$ the $n\times n$ matrix whose $(i,j)^{\rm th}$ entry is $1$ and any other entry is $0$. Let $u\in\Omega$ and $f_u=\tr(a_1)\cdots\tr(a_r)$ for some $a_1,\ldots,a_r\in\N$. Given $a_1=z_1\cdots z_s$, $a_2=z_{s+1}\cdots z_{l}$, and so on, where $z_1,\ldots,z_l$ are letters, we set $Z_i=e_{i,i+1}$ for $1\leq i<s$ and $Z_s=e_{s,1}$. Similarly, we define $Z_i=e_{i,i+1}$ for $s+1\leq i<l$ and $Z_l=e_{l,s+1}$.
Considering $a_3,\ldots,a_r$, we define $Z_i$ for all $l<i\leq d$ as above. 

Note that  in the symplectic case $e_{ij}^{\ast}=e_{j+d,i+d}$ for $1\leq i,j\leq d$. Hence in both cases the result of substitutions 
$$x_{ij}(k)\to (i,j)^{\rm th} \text{ entry of }Z_k \;\;(1\leq k\leq d) $$ 
in $\Psi_n(f_w)$ is zero for $w\neq u$ and one for $w=u$. Since $f$ is a free relation, we have $\Psi_n(f)=0$. Thus we obtain $\al_u=0$; a contradiction.
\end{proof}

We now can prove Theorem~\ref{theo_free}:
\begin{proof}
Let $f$ be a non-zero free relation. Obviously, without loss of generality we can assume that $d$ is large enough. Then Corollary~\ref{cor_reduction} and Lemmas~\ref{lemma_multilin},~\ref{label_no_multilin} imply a contradiction.
\end{proof}

\begin{remark}\label{remark_Donkin}
In case $G=GL(n)$ we can repeat the proof of Theorem~\ref{theo_free} without reference    
to Lemma~\ref{lemma_linind2}, where the restriction $p\neq 2$ is essential. As the result, we obtain that there is no free relations for $R^{GL(n)}$ for an arbitrary $p$. 
\end{remark}

\begin{remark}\label{remark_conj}
Let $p=2$ and $G=\Sp(n)$. By straightforward calculations we can see that $\tr(AA^{\ast})=0$ for every $n\times n$ matrix $A$ over $R$. By part~b) of Lemma~\ref{lemma_subword00} and part~c) of Lemma~\ref{lemma_subword}, elements $\tr(a)\in\N_{\si}$ with $a\stackrel{c}{\sim} a^T$ are free relations. On the other hand, it is not difficult to see that $\si_2(x_ix_i^T)$ is not a free relation ($1\leq i\leq d$).   
\end{remark}

\section{Invariants of mixed representations of quivers}\label{section_mixed}

A {\it quiver} $\Q=(\Q_0,\Q_1)$ is a finite oriented graph, where $\Q_0$ ($\Q_1$, respectively) stands for the set of
vertices (the set of arrows, respectively). For an arrow $a$, denote by $a'$ its head
and by $a''$ its tail. 
We say that $a=a_1\cdots a_r$ is a {\it path} in $\Q$ (where $a_1,\ldots,a_r\in
\Q_1$), if $a_1''=a_2',\ldots,a_{r-1}''=a_r'$. 
The head of the path $a$ is $a'=a_1'$ and the tail is $a''=a_r''$.  A path
$a$ is called {\it closed} if $a'=a''$.

Given a {\it dimension vector} $\n=(\n_v\,|\,v\in\Q_0)$, we consider 
\begin{enumerate}
\item[$\bullet$] the space $H=\sum_{a\in\Q_1} \FF^{\n_{a'}\times \n_{a''}}\simeq \sum_{a\in\Q_1} \Hom(\FF^{\n_{a''}},\FF^{\n_{a'}})$;

\item[$\bullet$] the coordinate ring $R=\FF[x_{ij}^a\,|\,a\in\Q_1,\,1\leq i\leq \n_{a'},\,1\leq j\leq \n_{a''}]$ of $H$;

\item[$\bullet$] the $\n_{a'}\times \n_{a''}$ {\it generic} matrix $X_a=(x_{ij}^a)$ for every $a\in\Q_1$; 

\item[$\bullet$] the group $GL(\n)=\sum_{v\in\Q_0} GL(\n_v)$, acting on $H$ as the base change, i.e., 
$$g\cdot (h_a)=(g_{a'} h_a g^{-1}_{a''})$$
for $g=(g_v)\in GL(\n)$ and $(h_a)\in H$; this action induces the action of $GL(\n)$ on $R$. 
\end{enumerate}
Given a path $a=a_1\cdots a_r$ with $a_i\in\Q_1$, we write $X_a$ for $X_{a_1}\cdots X_{a_r}$. Donkin~\cite{Donkin94} proved that the algebra of {\it invariants} of representations of $\Q$ 
$$I(\Q,\n)=R^{GL(\n)}$$ 
is the subalgebra of $R$ generated by $\si_t(X_a)$, where $a$ is a closed path in $\Q$ and $1\leq t\leq \n_{a'}$. Moreover, we can assume that $a$ is {\it primitive}, i.e., is not equal to the power of a shorter closed path in $\Q$.

Let $\i:\Q_0\to\Q_0$ be an involution, i.e., $\i^2$ is the identical map, satisfying $\i(v)\neq v$ and $\n_{\i(v)}=\n_v$ for every vertex $v\in\Q_0$. Define 
\begin{enumerate}
\item[$\bullet$] the group $GL(\n,\i)\subset GL(\n)$ by 
$(g_v)\in GL(\n,\i)$ if and only if $g_v g_{\i(v)}^T=E$ for all $v$;

\item[$\bullet$] the {\it double} quiver $\Q^{\D}$ by $\Q_0^{\D}=\Q_0$ and  $\Q_1^{\D}=\Q_1\coprod \{a^T\,|\,a\in \Q_1\}$, 
where $(a^T)'=\i({a''})$, $(a^T)''=\i({a'})$ for all $a\in\Q_1$.  
\end{enumerate}
We set $X_{a^T}=X_a^T$ for all $a\in\Q_1$. Zubkov~\cite{ZubkovI} showed that the algebra of {\it invariants} of {\it mixed} representations of $\Q$
$$I(\Q,\n,\i)=R^{GL(\n,\i)}$$
is the subalgebra of $R$ generated by $\si_t(X_a)$, where $a$ is a closed path in $\Q^D$ and $1\leq t\leq \n_{a'}$. As above, we can assume that $a$ is primitive.  An example of mixed representations of a quiver is given at the end of the section. 

Let $\Q$, $\n$, $\i$ be as above. We write $\M(\Q,\i)$ for the set of all closed paths in $\Q^{\D}$ and $\N(\Q,\i)$ for the subset of primitive paths.  Given a path $a$ in $\Q^{\D}$, we define the path $a^{T}$ in $\Q^{\D}$ and introduce $\sim$-equivalence on $\M(\Q,\i)$ in the same way as in Section~\ref{section_intro}. Denote by $\M_{\FF}(\Q,\i)$ the vector space with the basis $\M(\Q,\i)$ and define $\N_{\si}(\Q,\i)$ in the same way as $\N_{\si}$ have been defined in Section~\ref{section_intro}. 
Consider a surjective homomorphism 
$$\Upsilon_{\n} : \N_{\si}(\Q,\i)\to I(\Q,\n,\i)$$ %
defined by  $\si_t(a) \to \si_t(X_{a})$, if $t\leq \n_{a'}$, and $\si_t(a) \to 0$ otherwise.
Its kernel $K_{\n}(\Q,\i)$ is the ideal of relations for $I(\Q,\n,\i)$. Elements of $K(\Q,\i)=\bigcap_{\m>0} K_{\m}(\Q,\i)$ are called {\it free} relations for $I(\Q,\n,\i)$.

Let $u,v\in\Q_0$ be vertices. We say that $a\in\M_{\FF}(\Q,\i)$ goes from $u$ to $v$ 
if $a=\sum_i\al_i a_i$, where $\al_i\in\FF$ and $a_i\in\M(\Q,\i)$ satisfies $a_i''=u$,  $a_i'=v$. If $a$ goes from $u$ to $u$, then we say that $a$ is {\it
incident} to $u$. 

\begin{lemma}\label{lemma_free_mixed} The ideal $K(\Q,\i)$ of free relations for $I(\Q,\n,\i)$ is zero for an arbitrary $p$.
\end{lemma}
\begin{proof} If $f$ is a free relation, then $f\in K_{\m}(\Q,\i)$ for a dimension vector $\m=(m,\ldots,m)$ of $\Q$, where $m>0$ is arbitrary. By part~b) of Lemma~\ref{lemma_subword00} and part~c) of Lemma~\ref{lemma_subword}, there does not exist a closed path $a$ in $\Q^{\D}$ with $a\stackrel{c}{\sim}a^T$. Hence $f$ does not contain a summand with a multiplier $\si_t(a)$, where $a$ is a closed path in $\Q^{\D}$ with $a\stackrel{c}{\sim} a^T$ and $t>0$. 

Exactly in the same way as we proved Theorem~\ref{theo_free} for $G=O(m)$, we can show that $f=0$. Here we do not use the first part of Lemma~\ref{lemma_linind2}, which holds for $p=2$, but we only need Lemma~\ref{lemma_linind1}, which holds for an arbitrary $p$.    
\end{proof}

Let us recall that the definition of $\sigma_{t,r}$ can be found in Section~3 of~\cite{Lopatin_Orel}.

\begin{theo}\label{theo_mixed} The ideal of relations $K_{\n}(\Q,\i)$ for $I(\Q,\n,\i)\simeq  \N_{\si}(\Q,\i)/K_{\n}(\Q,\i)$ is generated by 
$$\si_{t,r}(a,b,c)\in\N_{\si}(\Q,\i),$$ 
where $t+2r>\n_v$ ($t,r\geq0$), $a,b,c\in\M_{\FF}(\Q,\i)$, $a$ is incident to some vertex $v\in\Q_0$, $b$ goes from $\i(v)$ to $v$, $c$ goes from $v$ to $\i(v)$.
\end{theo}
\begin{proof} As in~\cite{ZubkovII}, we denote by $J(\Q,\i)$ the inverse limit of algebras $$\{I(\Q,\n(1),\i),\varphi_{n(1),n(2)}\,|\,\n(1)\geq\n(2)\},$$%
where $\varphi_{n(1),n(2)}:I(\Q,\n(1),\i)\to I(\Q,\n(2),\i)$ is the natural epimorphism. It is not difficult to see that $J(\Q,\i)\simeq \N_{\si}(\Q,\i)/K(\Q,\i)$.  Lemma~\ref{lemma_free_mixed} implies that $J(\Q,\i)\simeq \N_{\si}(\Q,\i)$. By Theorem~2 of~\cite{ZubkovII}, the kernel of the natural epimorphism $J(\Q,\i)\to I(\Q,\n,\i)$ is generated by elements from Theorem~\ref{theo_mixed}. 
\end{proof}

\begin{example}\label{ex_forms}
Let $(\cdot,\cdot)_1,\ldots,(\cdot,\cdot)_r$ be bilinear forms on $V=\FF^n$ defined by $n\times n$ matrices $A_1,\ldots,A_r$ and $\LA\cdot,\cdot\RA_1,\ldots,\LA\cdot,\cdot\RA_s$ be bilinear forms on the dual space $V^{\ast}$ defined by $n\times n$ matrices $B_1,\ldots,B_s$.  Then $G=GL(n)$ acts on the space 
$$H=\bigoplus_{k=1}^r \FF^{n\times n}\oplus \bigoplus_{l=1}^s \FF^{n\times n}$$ 
of the above mentioned bilinear forms as base change: 
$$g\cdot (A_1,\ldots,A_r,B_1,\ldots,B_s)=(g A_1 g^{T},\ldots,g A_r g^{T},g^{-T}B_1 g^{-1},\ldots,g^{-T}B_s g^{-1}),$$
where $g^{-T}$ stands for $(g^T)^{-1}$. This action induces the action of $GL(n)$ on the coordinate ring 
$$\FF[H]=\FF[x_{ij}(k),\,y_{ij}(l),\,|\,1\leq i,j\leq n,\,1\leq k\leq r,\,1\leq l\leq s].$$
Denote generic matrices by $X_k=(x_{ij}(k))$ and $Y_l=(y_{ij}(l))$. Let $\Q$ be the following quiver
$$\vcenter
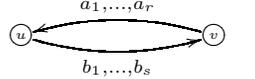
{
\xymatrix@C=1cm@R=1cm{ %
\vtx{u}\ar@/^/@{<-}[rr]^{a_1,\ldots,a_r} \ar@/_/@{->}[rr]_{b_1,\ldots,b_s}&&\vtx{v}\\
}} \quad,
$$
where there are $r$ arrows from $v$ to $u$ and $s$ arrows in the opposite direction,  
$\i(u)=v$, and $\n=(n,n)$. Then the algebra of invariants $\FF[H]^{GL(n)}$ is isomorphic to $I(\Q,\n,\i)$. By the above mentioned result of Zubkov~\cite{ZubkovI}, $\FF[H]^{GL(n)}$ is generated by $\si_t(Z_1\cdots Z_m)$, where $1\leq t\leq n$ and $Z_i$ is one of the following products:
$$X_k Y_l,\,X_k^T Y_l,\,X_k Y_l^T,\,X_k^T Y_l^T\; (1\leq k\leq r,\,1\leq l\leq s).$$  
Relations between these generators are described by Theorem~\ref{theo_mixed}.
\end{example}


\section*{Acknowledgements}
The final version of the paper was prepared during author's visit to Max-Planck Institute for Mathematics in Bonn. The author is grateful to MPIM for this support.  The research was also supported by RFFI 10-01-00383 and DFG.

\end{document}